\documentclass[11pt]{amsart}

\usepackage[a4paper,hmargin=3.5cm,vmargin=3.5cm]{geometry}
\usepackage{amsfonts,amssymb,amscd,amstext}
\usepackage{graphicx}
\usepackage[dvips]{epsfig}

\usepackage{fancyhdr}
\usepackage{times}

\usepackage{enumerate}
\usepackage{titlesec}
\usepackage{mathrsfs}

\def\r{\mathbb{R}}

\def\s{\mathbb{S}}

\titleformat{\section}
{\filcenter\bfseries\large} {\thesection{.}}{0.2cm}{}
\titleformat{\subsection}[runin]
{\bfseries} {\thesubsection{.}}{0.15cm}{}[.]
\titleformat{\subsubsection}[runin]
{\em}{\thesubsubsection{.}}{0.15cm}{}[.]

\usepackage[up,bf]{caption}

\newtheorem{theorem}{Theorem}[section]
\newtheorem{proposition}[theorem]{Proposition}

\newtheorem{lemma}[theorem]{Lemma}

\newtheorem{remark}[theorem]{Remark}

\theoremstyle{definition}

\numberwithin{equation}{section}
\numberwithin{figure}{section}

\usepackage{color}

\begin{document}

\fancyhead[CO]{Stable capillary hypersurfaces} 
\fancyhead[CE]{A. Ainouz$\,$  and$\,$ R. Souam} 
\fancyhead[RO,LE]{\thepage} 

\thispagestyle{empty}

\vspace*{1cm}
\begin{center}
{\bf\LARGE Stable capillary hypersurfaces in a half-space or  a slab}

\vspace*{0.5cm}

{\large\bf Abdelhamid Ainouz $\;$ and$\;$ Rabah Souam}
\end{center}


\noindent A. Ainouz

\noindent Department of Mathematics,
University of Sciences and Technology Houari Boumedienne
Po Box 32 El Alia,
Bab Ezzouar 16111
Algiers, Algeria

\noindent e-mail: {\tt aainouz@usthb.dz}

\vspace*{0.1cm}

\noindent R. Souam

\noindent Institut de Math\'{e}matiques de Jussieu-Paris Rive Gauche,   UMR 7586, B\^{a}timent Sophie Germain,  Case 7012, 75205  Paris Cedex 13, France.

\noindent e-mail: {\tt rabah.souam@imj-prg.fr}

\vspace*{0.1cm}

\vspace*{1cm}

\begin{quote}
{\small
\noindent {\bf Abstract}\hspace*{0.1cm}\vspace*{0.2cm} 
We study stable immersed capillary hypersurfaces in a domain $\mathcal B$ which is either a half-space or a  slab in the Euclidean space $\r^{n+1}.$ We prove that such a hypersurface $\Sigma$ is rotationally symmetric  in the following cases: 
\begin{enumerate}
\item $n=2$, $\mathcal B$ is a slab  and $\Sigma$ has genus zero,
\item $n\geq 2$, $\mathcal B$ is a slab, the angle of contact is $\pi/2$ and each component of $\partial\Sigma$ is embedded,
\item $n\geq 2,$ $\mathcal B$ is a half-space, the angle of contact is $<\pi/2$ and each component of  $\partial\Sigma$ is embedded.
\end{enumerate}
Moreover, in case (2), if not a right circular cylinder, $\Sigma$  has to be graphical over a domain in $\partial\mathcal B.$ In case (3), $\Sigma$ is a spherical cap.

\vspace*{0.2cm}

\noindent{\bf Keywords}\hspace*{0.1cm} Capillary hypersurfaces, constant mean curvature hypersurfaces, stability.

\vspace*{0.2cm}

\noindent{\bf Mathematics Subject Classification (2010)}\hspace*{0.1cm} 53A10, 49Q10, 53C42, 76B45.}

\end{quote}


\section{Introduction}\label{sec:intro}

Consider a closed domain $\mathcal B$ with smooth boundary in the Euclidean space $\r^{n+1}$. A capillary hypersurface in $\mathcal B$ is a  compact CMC hypersurface (i.e  with constant mean curvature) with non-empty boundary, which is contained in $\mathcal B$
and   meets  the frontier $\partial \mathcal B$  at a constant angle  along its boundary. When the angle of contact  is $\pi/2$, that is, when the hypersurface  is orthogonal to $\partial \mathcal B$, it is said to be a CMC hypersurface with free boundary. 

Capillary surfaces  in $\r^3$ model incompressible liquids inside  containers in the absence of gravity. Indeed, 
the free surface  of the liquid (locally) minimizes   an energy functional under a volume constraint. More precisely, given an angle $\theta\in(0,\pi),$ for a compact surface $\Sigma$ inside $\mathcal B$ such that $\partial \Sigma \subset \partial \mathcal B$ and $\partial \Sigma$ bounds a compact domain $W$ in
$\partial \mathcal B,$ the energy of $\Sigma$ is  by definition the quantity
\[
\mathcal E(\Sigma):= {\rm Area} (\Sigma) -\cos\theta \,  {\rm Area (W)}.
\]
The stationary surfaces of $\mathcal E$ for variations preserving the enclosed volume are precisely the 
CMC surfaces  which make a constant angle $\theta$ with  $\partial \mathcal B.$ In the physical interpretation, $\Sigma$ represents the liquid-air 
interface and $W$ the region of the container wetted by the liquid and $c:=\cos\theta$ is a physical constant. For more information  on capillary surfaces we refer to  \cite{Fi} and \cite{langbein}.

From the geometrical viewpoint, capillary hypersurfaces  raise several interesting questions. Given a domain $\mathcal B\subset \r^{n+1}$, one would like to know whether there exist immersed capillary hypersurfaces in $\mathcal B$ of a given topology and 
to characterize, in particular when $\mathcal B$ has symmetries, the embedded ones. These are quite difficult problems in general. For instance, even in the case of the unit ball in $\r^3$, these issues are still not completely understood. They  are related to problems in spectral geometry when the surfaces are minimal \cite{fraser-schoen, souam}. In a pioneering work, Nitsche \cite{nitsche} has shown that an immersed capillary disk in the unit ball in $\r^3$ has to be a spherical cap or a flat disk. Recently, Fraser and Scheon \cite{fraser-schoen} proved the existence of embedded minimal surfaces of genus zero with free boundary in the unit ball of $\r^3$ with arbitrarily many boundary components. In the case of a closed slab or a closed half-space of $\r^{n+1},$ one can apply Alexandrov's reflection technique to 
prove that in the embedded case, capillary hypersurfaces have rotational symmetry around an axis orthogonal to the boundary of the domain \cite{wente}. Wente \cite{We}  has also proven   the existence of immersed, non embedded, capillary surfaces of annular type in a closed slab and in the unit ball in $\r^3$ which are not pieces of Delaunay surfaces. In the case of a wedge in $\r^3,$ it has been shown that pieces of spheres are the only embedded capillary surfaces of annular type \cite{Pa, Mc}. Recently, in \cite{alarcon-souam}, it was shown there exists  a large family  of embedded capillary surfaces in  convex polyhedral domains in $\r^3.$ The surfaces have genus zero and one boundary component lying on each face of the polyhedra.   

Motivated by the physical interpretation, it is natural to restrict one's attention to  the capillary hypersurfaces  which are stable, that is, those which minimize the energy $\mathcal E$ up to second order under the volume constraint. This has been addressed for the case of the unit ball  \cite{li-xiong, ros-vergasta, ros-souam} and for the case of a wedge or a half-space in $\r^{n+1}$ 
 \cite{choe-koiso, marinov}, and also for more general convex bodies \cite{ros1, ros-vergasta}.

In this paper we will be  concerned  with  stable immersed capillary hypersurfaces in a domain $\mathcal B$ which is either a closed slab or a closed half-space in $\r^{n+1}.$ For simplicity, we will refer to the latter as slabs and hyperplanes, meaning implicitly they are closed. We will prove that such a hypersurface  is  rotationally invariant around an axis  orthogonal to 
$\partial\mathcal B$  in the following cases:
\begin{enumerate} 
\item $n=2,$ $\mathcal B$ is a slab and the surface has genus zero (Theorem \ref{th:stable genus 0}),
\item $n\geq 3,$ $\mathcal B$ is a slab, the hypersurface has free boundary, that is, the angle of contact is  $\pi/2$  and each boundary 
component of the hypersurface is embedded . Furthermore, if not a right circular cylinder, the hypersurface has to be a graph over a domain in $\mathcal B$ (Theorem \ref{thm:free boundary}),
\item $n\geq 3,$ $\mathcal B$ is a half-space, the angle of contact either is $\pi/2$, or is $<\pi/2$  and each boundary 
component of the hypersurface is embedded. Moreover the hypersurface is then a spherical cap (Theorem \ref{thm:halfspace}).
\end{enumerate}


\newpage

\section{Preliminaries}\label{sec:prelim}

\subsection{The variational problem and stability}\label{sub:variational}

Consider a domain $\mathcal B$  with smooth boundary in an oriented Riemannian 
manifold $(M, \langle,\rangle)$ of dimension $n+1$. We let $\bar N$ denote the unit normal to $\partial \mathcal B$ 
pointing outwards $\mathcal B.$ The orientation of $M$ induces an orientation on $\partial \mathcal B$ in the usual way:  a local orthonormal frame $\{\epsilon_1,\ldots,\epsilon_n\}$ on $\partial \mathcal B$ is positively oriented if the local frame 
$\{\bar N, \epsilon_1,\ldots,\epsilon_n\}$ has positive orientation in $M.$ 

  In all what follows,  unless otherwise stated,  $\Sigma$ will denote a compact  orientable smooth manifold of dimension $n$ with  non empty boundary, $\partial \Sigma,$ and   $\psi:  \Sigma \to \mathcal B$ an immersion which is
  smooth on the interior of $\Sigma$ and of class $\mathcal C^2$ up to the boundary, and which is {\it proper}, that is, verifying  $\psi(\text{ int}\, \Sigma)\subset \text {int}\, \mathcal B$ and $\psi(\partial\Sigma)\subset\partial \mathcal B.$
  
  Fix a global unit normal $N$ to $\Sigma$ along $\psi.$ This determines an orientation on $\Sigma$ as above.
Denote by $\nu$ the exterior unit normal to $\partial \Sigma$ in $\Sigma$. Again this induces an orientation on 
$\partial \Sigma$. 
Let now  $\bar \nu$ be the unit normal to $\partial \Sigma$ in $\partial \mathcal B$ compatible with this orientation. Otherwise said, $\bar\nu$ is such that the 
the bases $\{N,\nu\}$ and $\{\bar N, \bar\nu\}$ have the same orientation in $(T\partial\Sigma)^\perp.$

By a variation of the immersion $\psi$ we mean a differentiable map $\Psi: (-\epsilon,\epsilon) \times \Sigma \to M$ such that $\psi_t:\Sigma\to M, t\in (-\epsilon,\epsilon),$ defined by $\psi_t(p)=\Psi(t,p), p\in\Sigma,$ is an immersion  for each $t\in(-\epsilon, \epsilon)$ and $\psi_0=\psi.$  

If $\psi$ is a proper immersion into $\mathcal B,$ then  a variation $\Psi$ is said to be admissible if   $\psi_t$ is a proper immersion into $\mathcal B$ for each  for each $t\in(-\epsilon, \epsilon).$
 In the sequel we will consider only admissible variations.
 
We let $\xi(p)= \frac{\partial \Psi}{\partial t}(0,p), p\in \Sigma,$  be the variation vector field of $\Psi.$  The {\it  volume function} $V:(-\epsilon,\epsilon)\to \r$ is by definition
\[ 
V(t)=\int_{[0,t]\times\Sigma} \Psi^\ast \Omega,
\]
where $\Omega$ is the  volume form of $M.$ The deformation is {\it volume-preserving} if $V(t)=0$ for each $t\in(-\epsilon,\epsilon).$ 

For $t\in(-\epsilon, \epsilon),$ we  denote by $A(t)$  the $n-$dimensional volume of 
$\Sigma$ for the metric induced by $\psi_t$. We also consider the {\it wetted area function} $ W(t):(-\epsilon,\epsilon)\to \r$ defined by
\[ 
 W(t)=\int_{[0,t]\times\partial\Sigma} \Psi^\ast \omega,
\]
where $\omega$ is the volume  form  of $\partial \mathcal B.$

Let  $\theta\in (0,\pi)$ be a given number, the energy functional $\mathcal E: (-\epsilon, \epsilon)\to \r$ is defined as follows
\[
\mathcal E(t)=  A(t) -\cos\theta \, W(t).
\]

 The first variation formulae for the energy and volume functions are as follows (cf. \cite{barbosa-do carmo})
 
 \begin{equation}\label{energyvariation}
  \mathcal E'(0)=-n\int_{\Sigma} H f  \,d\Sigma +\int_{\partial \Sigma} \langle \xi, \nu-\cos\theta\, \bar\nu\rangle \,d(\partial\Sigma)
  \end{equation}
 \begin{equation}\label{volumevariation}
 V'(0)=\int_{\Sigma} f \,d\Sigma.
\end{equation}

 Here $d\Sigma$ (resp. $d({\partial \Sigma})$) denotes the $n-$volume element  on 
 $\Sigma$  (resp. the $(n-1)-$volume element on ${\partial \Sigma}$), $H$ is the mean curvature of the immersion 
 $\psi$ computed with respect to the unit normal $N$ and $f=\langle \xi, N\rangle.$ 
 
Extending a result in \cite{barbosa et al}, we have the following infinitesimal characterization 
of admissible volume preserving variations. 

 \begin{proposition}\label{prop:deformations}
 Assume the proper immersion  $\psi$  into  $\mathcal B$  is transversal to $\partial\mathcal B$. Let $f$ be a smooth function on $\Sigma$ satisfying $\int_{\Sigma} f d\Sigma =0.$ Then there exists an admissible volume-preserving variation $\Psi$ of $\psi$ in $\mathcal B$ such that $f=\langle \frac{\partial\Psi}{\partial t}|_{0},N\rangle$. 
 \end{proposition}
 
 \begin{proof} 
 
 Start with an admissible variation $F: (-\epsilon,\epsilon)\times \Sigma\to \mathcal B, \epsilon>0,$ of the immersion $\psi$ 
such that  $F$ is a local diffeomorphism. An example of such a variation can be constructed as follows. Consider an open
neighborhood $W$ of $\psi(\partial \Sigma)$ in $M$ which is diffeomorphic to a product $(-\delta,\delta)\times U,$ where $U\subset \partial\mathcal B$ is an open  relatively compact neighborhood of $\psi(\partial\Sigma)$.  Endow $W$ with the product 
metric and extend it to a Riemannian metric on $M$. For this new metric, $U$ is totally geodesic. It suffices to take the mapping $F: (-\epsilon,\epsilon)\times \Sigma \to M$ defined for $\epsilon >0$ small enough and $(t,p)\in
(-\epsilon,\epsilon)\times \Sigma$ 
by 
 $F(t,p)={\text{exp}}_{\psi(p)} (tN(p)). $ Here exp is the exponential map and $N$  a unit normal field to $\Sigma$  for the new metric. 

We now endow the manifold $\widetilde\Sigma :=(-\epsilon,\epsilon)\times \Sigma$ with the pull-back metric $F^\ast \langle\,\, , \,\rangle$, of the original metric on $M$, which we denote also by $\langle\,\, , \,\rangle$. It is enough to prove the result for the hypersurface $\{0\}\times \Sigma$ ,
 in the domain $\widetilde\Sigma$ for this metric. Indeed, if $\Phi$ is an admissible volume-preserving  variation  of $\{0\}\times \Sigma$ in $\widetilde\Sigma$  then $F\circ\Phi$ is an admissible variation of $\psi$ in $\mathcal B.$ Furthermore if $N$ is a unit field normal to $\{0\}\times \Sigma$ in $\widetilde\Sigma $ then $dF(N)$ is a unit field normal to $\Sigma$ in $M$ and $\langle \frac{\partial (F\circ \Phi)}{\partial t}|_0,dF(N)\rangle= \langle dF(\frac{\partial \Phi}{\partial t}|_0), dF(N)\rangle=\langle \frac{\partial \Phi}{\partial t}|_0, N\rangle$.  Moreover,
 
 $$ \int_{[0,t]\times\Sigma} (F\circ\Phi)^\ast \Omega = \int_{[0,t]\times\Sigma} \Phi^\ast (F^\ast \Omega)= \int_{[0,t]\times\Sigma} \Phi^\ast d\widetilde\Sigma,$$
 so that the volume functions of the deformations $F\circ\Phi$ and $\Phi$ coincide and hence  $F\circ\Phi$ is  volume-preserving too. 
 
 We now prove the result in $\widetilde \Sigma$. For each point $p\in \partial\Sigma$, let $N_0(p)=N(p)-\langle N(p), \bar N (p)\rangle \bar N(p)$ be the projection of $N(p)$ on $T_p(\partial \widetilde\Sigma)$; note that $\partial \widetilde\Sigma=(-\epsilon,\epsilon)\times \partial\Sigma$. We now consider the 
 vector $w(p)= \frac{1}{\langle N(p), N_0(p)\rangle } N_0(p)-N(p)$ in $T_p(\{0\}\times \Sigma)$ which is well defined by the transversality assumption. 
 We can extend $w$ to a vector field on $\{0\}\times \Sigma$, still denoted $w.$ Set $z=w+N$, we can extend
 $z$, for instance in a trivial way using the product structure, to a vector field on 
 $\widetilde \Sigma$ which is tangent to $ \partial\widetilde \Sigma$ along $ \partial\widetilde \Sigma$. 
  Call  $\mathcal Z$ such an extension. By construction $\mathcal Z$ satisfies $\langle \mathcal Z, N\rangle =1$ on 
  $\{0\}\times \Sigma$. 
Let $(\phi_t)_{|t|<\delta}, \delta>0,$ denote the local flow of $\mathcal Z$ and consider  the map:
  $\Phi: (-\delta,\delta)\times \Sigma\to \widetilde\Sigma,$ defined by $\Phi(t,p)=\phi_t(p),$
   for $(t,p)\in  (-\delta,\delta)\times\Sigma$. Let now $u:(-\epsilon_0,\epsilon_0)\times \Sigma\to \widetilde \Sigma,\,\, \epsilon_0>0,$
   be a differentiable function and  define a variation $\mathcal X: (-\epsilon_0,\epsilon_0)\times \Sigma\to \widetilde \Sigma$ as follows:
    \[ (t,p)\to  \mathcal X(t,p)= \Phi(u(t,p),p),\qquad t\in (-\epsilon_0,\epsilon_0), \, p\in \Sigma
     \]
  We have $\frac{\partial\mathcal X}{\partial t}|_{0}=\frac{\partial u}{\partial t}|_{0}\, \mathcal Z$ and so 
 $\langle \frac{\partial\mathcal X}{\partial t}|_{0}, N\rangle =\frac{\partial u}{\partial t}|_{0} $.  Clearly $\mathcal X$ is an admissible variation. We will see that we can choose the function $u$ to fulfill the required properties. To compute the volume function of $\mathcal X$ we  write: $\mathcal X= \Phi\circ \Psi$ where $\Psi:(-\epsilon_0,\epsilon_0)\times \Sigma \to \widetilde \Sigma$ is such that $\Psi(t,p)= (u(t,p),p).$The volume function associated to $\mathcal X$ is 
  $$V(t)=\int_{[0,t]\times\Sigma} \mathcal X^\ast \Omega = \int_{[0,t]\times\Sigma} \Psi^\ast (\Phi^\ast \Omega).$$
  Set $\Phi^\ast \Omega= E\,  dt\wedge d\Sigma$ where $E:(-\epsilon_0,\epsilon_0)\times \Sigma \to \widetilde \Sigma$ is a smooth function. Then $\Psi^\ast (\Phi^\ast \Omega)(t,p)= E(u(t,p),p) \frac{\partial u}{\partial t} (t,p) \,dt\wedge d\Sigma$ and so
  $$V(t)= \int_{\Sigma} \left( \int_0^t E(u(t,p),p) \frac{\partial u}{\partial t} (t,p) dt\right) d\Sigma.$$
  Now, we choose $u$ to be the function which solves the following initial value problem:
  
  $$ \frac{\partial u}{\partial t} (t,p)=\frac{f(p)}{E(u(t,p),p)},\qquad
u(0,p)=0,\quad  \text{for each }p\in\Sigma.
$$
  Then 
  $$V(t)= 0, \,\, {\text {for all}} \, t\in(-\epsilon_0,\epsilon_0),$$
  so that $\mathcal X$ is volume-preserving. 
  
  To finish, let us we check that $\frac{\partial u}{\partial t}(0,p)=f(p)$ for all $p\in\Sigma.$ To see this, take a local orthonormal frame $\{\epsilon_1,\ldots,\epsilon_n\}$ in a neighborhood of $p\in\Sigma$, then 
  $E(0,p)=\Omega\left(d\Phi_{(0,p)}(\epsilon_1),\ldots , d\Phi_{(0,p)}(\epsilon_n),\frac{\partial \Phi}{\partial t}(0,p)\right)= \Omega\left(\epsilon_1,\ldots , \epsilon_n,\mathcal Z(p)\right)=\Omega\left(\epsilon_1,\ldots , \epsilon_n,w(p)+N(p)\right)=\Omega\left(\epsilon_1,\ldots , \epsilon_n,N(p)\right)=1,$ where we used the fact that 
  $\Phi(0,q)=q$ for all $q\in \Sigma$ and that $w(p)$ is tangent to $\Sigma$. 
  
   Thus $\frac{\partial u}{\partial t} (0,p)=f(p)$ 
  and so $\langle \frac{\partial\mathcal X}{\partial t}|_{0}, N\rangle = f$ as required.
   \end{proof}

The proper immersion $\psi:\Sigma\to\mathcal B$ is said to be a capillary   immersion into $\mathcal B$ if $\psi$ is a critical point of the  energy functional $\mathcal E$ for admissible volume-preserving 
 variations.  It follows, from formulae (\ref{energyvariation}) and (\ref{volumevariation})  that  $\psi$ is capillary if and only if it has constant mean curvature and the angle $\in[0,\pi] $ determined by $\nu$ and $\bar \nu,$ called the angle of contact of $\Sigma$ with $\partial \mathcal B,$ is constant and equal to $\theta$ along $\partial \Sigma.$ This amounts to the same to saying that the angle between $N$ 
 and $\bar N$ is constant and equal to $\theta.$  Note that the value $\theta$ of the angle of contact depends on the chosen  unit normal $N$.  With the opposite choice, $-N$,  the angle would be  $\pi-\theta.$ 
 
 {\bf Convention.} {\it In the sequel we will always take the unit normal $N$ so that the (constant) mean curvature $H$ is $\geq 0.$ So, when 
 $H$ is not 0, we orient the hypersurface by its mean curvature vector $\vec H$ and the contact angle is the one between $\vec H$ and the exterior unit normal $\bar N$ to $\partial\mathcal B.$}
 
Given a function $f\in\mathcal C^{\infty}(\Sigma)$ satisfying $\int_{\Sigma} f\,d\Sigma=0$ and an admissible volume-preserving variation $\Psi$ with
$f=\langle \frac{\partial\Psi}{\partial t}|_{0},N\rangle $, we have (cf. \cite{ros-souam})
 \[ \mathcal E^{''}(0)=  -\int_{\Sigma} f\left(\Delta f +(|\sigma|^2+\text{Ric}(N)) f\right)\,d\Sigma + \int_{\partial\Sigma} 
 f\left( \frac{\partial f}{\partial \nu} - q\, f\right)\, d(\partial\Sigma),
 \]
where $\Delta$ is the Laplacian for the metric induced by $\psi$ on $\Sigma$  and $\sigma$ is the second fundamental form of $\psi,$ Ric is the Ricci curvature of $M$ and  
\[ q= \frac{1}{\sin\theta} \text{II} (\bar\nu,\bar\nu)+ \cot\theta \,\sigma(\nu,\nu). 
\]
Here II denotes the second fundamental form of $\partial\mathcal B$ associated to the unit normal $-\bar N.$

A capillary immersion $\psi$ is called stable if $\mathcal E^{''}(0)\geq 0$ for all admissible volume-preserving variations. Let $\mathcal F=\{ f\in H^1(\Sigma), \int_{\Sigma} f\,d\Sigma =0\},$ where $H^1(\Sigma)$ is the first Sobolev space of $\Sigma.$ The index form $\mathcal I$ of $\psi$ is the symmetric bilinear form defined on $H^1(\Sigma)$ by
\[ \mathcal I (f,g)=\int_{\Sigma}\left( \langle \nabla f, \nabla g\rangle -(|\sigma|^2+\text{Ric}(N)) fg\right) d\Sigma-
 \int_{\partial\Sigma}q\, fg\,d(\partial\Sigma),
\]
where $\nabla$  stands for the gradient for the metric induced by $\psi.$ It follows from  Proposition \ref{prop:deformations} and a standard density argument that  $\psi$ is capillarily stable if and only if $\mathcal I(f,f)\geq 0$ for all $f\in\mathcal F.$ 

A function $f\in \mathcal F$ is said to be a Jacobi function of $\psi$ if it lies in the kernel of $\mathcal I$, that is, if 
$\mathcal I(f,g)=0$ for all $g\in \mathcal F.$ By standard arguments, this is equivalent to saying that $f\in \mathcal C^{\infty}(\Sigma)$ and satisfies
\begin{align*}
\Delta f+(|\sigma|^2+\text{Ric}(N)) f &= \text{constant \quad on }\, \Sigma\\
\frac{\partial f}{\partial \nu} &=q\,f\quad\,\, \quad\quad\text{on}\, \partial \Sigma
\end{align*}

More generally, we may assume that the contact  angle is constant along each component of $\partial \Sigma$. If $\Gamma_1,\dots,\Gamma_k$ denote the connected components of $\partial\Sigma$ and $\theta_1,\dots,\theta_k$ are given 
 angles in $(0,\pi)$, the capillary hypersurfaces in $\mathcal B$ with contact angle $\theta_i$ along $\Gamma_i$ for each $ i=1,\ldots,k$ are the critical points for admissible volume-preserving variations of the energy functional 
 
  \[
 \mathcal E(t) = A(t)-\sum_{i=1}^k \cos\theta_i\, W_i(t)
 \]
where, $W_i(t)$ denotes  the  wetted area function corresponding to $\Gamma_i, i=1,\ldots, k$.
The first and second variations formulae as well as  the previous discussion are valid, with obvious modifications, in this more general setting. 

A fact that we will use repeatedly is the following: 

\begin{lemma}\label{lem:umbilicity}
 Suppose that $\partial\mathcal B$ is totally umbilical.  Let $\psi$ be a capillary immersion into $\mathcal B.$ Then, the unit outwards normal $\nu$ along $\partial \Sigma$ in $\Sigma$ is a principal direction of $\psi.$ 
\end{lemma}
So, if $D$ denotes the Levi-Civita connection on the ambient manifold, we have
$$ D_{\nu}N=-\sigma(\nu,\nu)\,\nu.$$
\begin{proof}
It suffices to show that $\sigma(\nu,x)=0$ for any $x\in T(\partial\Sigma).$ Along $\partial\Sigma$ we have $N=\cos\theta\,\bar N-\sin\theta\, \bar\nu$, where $\bar\nu$ is the unit normal to $\partial\Sigma$ in $\partial\mathcal B$ introduced above. So: $D_{x} N=\cos\theta\,D_x{\bar N} -\sin\theta\,D_x{\bar \nu}.$ Let  II denote, as above,  the second fundamental form of $\partial\mathcal B$ associated to $-\bar N.$ By hypothesis, $D_x{\bar N}=\text{II}(x,x)\,x$ is orthogonal to $\nu$. Furthermore, denoting by $\nabla$ the Levi-Civita connection on $\partial\mathcal B,$
again since $\partial\mathcal B$ is totally umbilical: $D_{x}{\bar\nu}=-\text{II}(x,\bar\nu)\,\bar N+\nabla_x{\bar\nu}=\nabla_x{\bar\nu}$ is tangent to $\partial\Sigma$.
It follows that  $\sigma(\nu,x)=-\langle D_x N,\nu\rangle=0.$ 
\end{proof}

 From now on, we will take  the ambient manifold $M$ to be  the Euclidean space $\r^{n+1}$ and 
 $\mathcal B$ will be  either a half-space or a slab. In this case, 
 Ric $\equiv 0$ and II$\equiv 0$.
 
\subsection{Some formulae for hypersurfaces in Euclidean spaces} 

We gather here some general  facts about hypersurfaces in  Euclidean spaces that we will use in the sequel. 
The first proposition contains well known formulae, see for instance \cite{barbosa-do carmo}. 
 \begin{proposition}\label{prop:equations}

Let $\psi:\Sigma\to \r^{n+1}, n\geq 2,$ be a $\mathcal C^2$-immersion in the Euclidean space $\r^{n+1}$ of a smooth orientable $n-$dimensional manifold  $\Sigma$, possibly with boundary, $N:\Sigma\to\s^n\subset\r^{n+1}$ a global unit normal of $\psi$,
 $H$ its mean curvature and $\sigma$ its second fundamental form.  Denote by $\Delta$ and div, respectively  the Laplacian and divergence operators for the metric induced by $\psi$. 
Then the following equations hold on $\Sigma$
\begin{enumerate}[\sf (i)]
\item $\Delta \psi = nH \,N$,
\item  div (${\psi}-\langle \psi,N\rangle N) = n +nH\langle \psi,N\rangle,$
\item for any constant vector field $\vec a$ on $\r^{n+1},$ div $({\vec a}-\langle \vec a, N\rangle N)= n H \langle \vec a, N\rangle$. 

\noindent Moreover, if the mean curvature $H$ is constant, then 
\item $\Delta \langle \psi, N\rangle + |\sigma|^2 \langle \psi, N\rangle = -nH,$
\item $\Delta N +|\sigma|^2 N = \vec 0.$
\end{enumerate}
\end{proposition}

A second result we will need  is the following useful fact of independent interest.

\begin{proposition}\label{prop:normalintegral}
Let $\psi:\Sigma\to \r^{n+1}$ be a  $\mathcal C^1$-immersion in the Euclidean space $\r^{n+1}$ of a smooth compact orientable $n-$dimensional manifold $\Sigma$, possibly with boundary. Let $N:\Sigma \to \s^n\subset\r^{n+1}$ be a global unit normal of  $\psi$ and $\nu$ the unit outward conormal to $\partial\Sigma$ in $\Sigma$. Then 
\begin{equation}\label{eq:normal}
n\int_{\Sigma} N \, d\Sigma=  \int_{\partial\Sigma} \{ \langle \psi,\nu\rangle N - \langle \psi,N\rangle \nu\} \,d(\partial\Sigma)
\end{equation}
where $d\Sigma$ and $d(\partial\Sigma)$ denote the volume elements  of $\Sigma$ and $\partial\Sigma,$ respectively.

In particular, if $\Sigma$ has no boundary, then
\begin{equation*}
\int_{\Sigma} N \, d\Sigma= \vec{0}.
\end{equation*}

\end{proposition}

\begin{proof}
We prove the result for smooth immersions, the general case follows by approximation. Let $\vec{a}$ be a constant vector field on $\r^{n+1}.$  Consider the following vector field on $\Sigma $
$$ X=  \langle \vec a, N\rangle \psi^T -  \langle \psi,N\rangle {\vec a}^T, $$
where,  $\psi^{T}={\psi}-\langle \psi,N\rangle N$  (resp. ${\vec a}^T= {\vec a}-\langle \vec a, N\rangle N$ ) is the projection of $\psi$ (resp. of $\vec a$) on $T\Sigma$. 
Using  (ii) and (iii)  of Propositions \ref{prop:equations} and denoting by $D$ the usual differentiation in $\r^{n+1}$, we compute the divergence of $X$:
\begin{align*} {\text {div}}\, X&= \langle \vec a, N\rangle \,{\text {div}}\, \psi^T + \langle \vec a, D_{\psi^T} N\rangle - \langle \psi, N\rangle \text{div}\, {\vec a}^T - \langle {\vec a}^T, N\rangle -\langle \psi, D_{{\vec a}^T} N\rangle \\
&=   n  \langle \vec a, N\rangle + nH\langle \psi,N\rangle \langle \vec a, N\rangle  +  \langle {\vec a}^T, D_{\psi^T} N\rangle-n H
\langle {\vec a}, N\rangle \langle \psi,N\rangle  -\langle \psi^T, D_{{\vec a}^T} N\rangle \\
&= n\langle \vec a, N\rangle,
\end{align*}
 where we used that $\langle {\vec a}^T, D_{\psi^T} N\rangle=\langle \psi^T, D_{{\vec a}^T} N\rangle=- \sigma({\vec a}^T, \psi^T),$ 
 $\sigma$ being, as above,  the second fundamental form of the immersion. Integrating on $\Sigma$ and using the divergence theorem we get
 \begin{equation*}
 n\int_{\Sigma}\langle  \vec a, N\rangle \, d\Sigma=  \int_{\partial\Sigma}  \{ \langle \vec a, N\rangle\langle \psi,\nu\rangle  - \langle \psi,N\rangle \langle\vec a, \nu \rangle\} d(\partial\Sigma).
  \end{equation*}
 Since this is true for any $\vec a$,   (\ref{eq:normal}) follows.
\end{proof}


{\section{Stable capillary surfaces  of genus zero in  a slab  in $\r^3$}\label{sec:proof}

We deal in this section with  stable capillary surfaces in a slab in $\r^3.$ In the free boundary case, that is, when  the angle of contact 
is $\theta=\pi/2$, it was proved by Ros that the surface has to be a right circular cylinder. This follows from the results in \cite{ros}. For general values of $\theta$, we will show that in the genus zero case, a stable capillary surface in a slab has to be a surface of revolution. In particular, the capillary annuli constructed by Wente \cite{We}  are unstable. The result is true in the more general case where the contact angles $\theta_1$ and $\theta_2$ 
with the 2 planes $\Pi_1$ and $\Pi_2$ bounding the slab, are not necessarily equal. Although we do not explicitely state it below, we do not need to assume the surfaces are contained in the slab, only the assumption on the boundary is relevant. The stability of embedded capillary of revolution connecting 2 parallel planes was studied by Vogel \cite{vogel2} and Zhou \cite{zhou}.

\begin{theorem}\label{th:stable genus 0}
Let $\psi$ be a   capillary immersion of a surface $\Sigma$ of genus zero in  a slab of $\r^3$ bounded by 2 parallel planes $\Pi_1$ and $\Pi_2$ and having constant contact angles $\theta_1$ and $\theta_2$ with  $\Pi_1$ and $\Pi_2,$ respectively. 

If $\psi$ is stable, then 
$\psi(\Sigma)$ is a surface of revolution around an axis orthogonal to $\Pi_1$.
\end{theorem} 

\begin{proof}
Up to isometries and a homothety of $\r^3,$  we may suppose the slab is bounded by the horizontal planes 
 $\Pi_1=\{x_3=0\}$ and $\Pi_2=\{x_3=1\}.$ Let  $\gamma$  be a connected component of  $\partial\Sigma$  such that $\psi(\gamma)$  lies on  the plane $\{x_3=0\}$ and consider in this plane  the  circumscribed circle
$\mathcal C$ about $\psi(\gamma)$.  We will show that $\psi(\Sigma)$ is a surface of revolution around the vertical axis passing through the center of $\mathcal C.$ 

We may assume that the center of $\mathcal C$ coincides with the origin. Let us  consider the  Jacobi function $u$ on $\Sigma$ induced by 
the rotations around the $x_3-$axis. More precisely, for $p\in \Sigma, u(p)=\langle \psi(p)\wedge e_3, N(p)\rangle,$ where $\wedge$ denotes the cross product on $\r^3.$ The function $u$ verifies
\begin{equation}\label{rotation}
\begin{cases}  \Delta u +|\sigma|^2 u = 0 \,\,\,\,\quad {\text {on}}\quad \Sigma\\
\, \,\,\,\quad \quad\quad\frac{\partial u}{\partial \nu} = q\, u \quad {\text{on}}\quad \partial \Sigma
 \end{cases}
\end{equation}

We will prove  that $u\equiv 0$ on $\Sigma.$ 

Suppose, by contradiction, $u$ is not identically zero. Then its nodal set $u^{-1}(0)$ in the interior of $\Sigma$ has the structure of a graph (cf. \cite{cheng}). Recall that a nodal domain of $u$ is a connected component of $\Sigma\setminus u^{-1}(0)$.
We will show that $u$ has at least 3 nodal domains by analyzing the set   $u^{-1}(0)\cap\gamma$.

We first note  that, because of the boundary condition satisfied by $u,$ if $p\in u^{-1}(0)\cap \partial\Sigma$ then
$\frac{\partial u}{\partial \nu}(p)=0.$ It follows that $u$ has to change sign in any neighborhood of $p\in u^{-1}(0)\cap \partial\Sigma.$  Indeed, otherwise, as
 $\Delta u=-|\sigma|^2 u,$ by the strong maximum principle (see \cite{gilbarg-trudinger}, Theorem 3.5 and  Lemma 3.4),
at such a point $p,$ we would have $\frac{\partial u}{\partial \nu}(p)\neq 0,$  unless $u$ is identically zero in a neighborhood of $p,$ but then, by the unique continuation principle of Aronszajn, $u$ would vanish everywhere on $\Sigma,$ contradicting our assumption. It follows that each such point  
lies on the boundary of at least 2 components of the  set $\{u\neq 0\}.$

If $p\in \partial\Sigma$ is a critical point of the function $|\psi|^2$ restricted to $\gamma$, then one can check that $u(p)=0$.
Now, we observe that by the choice of $\mathcal C,$ there are at least 3 points in  $u^{-1}(0)\cap \gamma.$ Indeed, it is a  known fact that $\mathcal C$ contains at least two points of $\psi(\gamma)$, see \cite{osserman}. This gives rise to 2 points in $u^{-1}(0)\cap \gamma.$ A third one is a point of $\gamma$ whose image by $\psi$ is a closest one to the origin. 

 Since by hypothesis $\Sigma$ is topologically a planar domain, using the above information and the Jordan curve theorem, it is easy to see this implies that $u$ has at least 3 nodal domains.

 Denote by $\Sigma_1$ and $\Sigma_2$ two of these components and consider the following function in the Sobolev space $H^1(\Sigma)$:
 \begin{equation*}
\widetilde u ={ \begin{cases}
\quad u \,\,\,\,\,\quad{\text{on}}\quad \Sigma_1\\
\alpha \, u \quad\quad {\text{on}}\quad \Sigma_2\\
\quad 0\qquad {\text{on}}\quad \Sigma\setminus (\Sigma_1\cup \Sigma_2)
 \end{cases}}
\end{equation*} 
 where $\alpha\in \r$ is chosen so that $\int_{\Sigma}\widetilde u\,dA =0.$
 Using (\ref{rotation}) we compute
 \begin{align*}
 \int_{\Sigma_1}\{ \langle \nabla\widetilde u, \nabla\widetilde u\rangle -|\sigma|^2 {\widetilde u}^2 \}dA
 &= \int_{\Sigma_1}\{ \langle  \nabla u, \widetilde\nabla u\rangle -|\sigma|^2 \,u{\widetilde u}\}dA\\
 &= -\int_{\Sigma_1} ( \Delta u+ |\sigma|^2 u) \widetilde u \,dA+\int_{\partial\Sigma_1} \widetilde u\frac{\partial u}{\partial \nu}\,ds\\
 &= \int_{{\partial \Sigma_1}\cap \partial\Sigma} q {\widetilde u}^2 \,ds
  \end{align*}
  Using a similar computation on $\Sigma_2,$ we deduce that 
  $\mathcal  I(\widetilde u,\widetilde u)=0.$
 As $\Sigma$ is stable, we conclude that  $\widetilde u$ is a Jacobi function. Indeed, the quadratic form on $\mathcal F$ associated to $\mathcal I$ has a minimum at $\widetilde u$ and so $\widetilde u$ lies in the kernel of $\mathcal I.$ However, $\widetilde u$ vanishes on a 
  non empty open set. By the unique continuation principle of Aronszjan, $\widetilde u$ has to vanish everywhere, which is a contradiction. 
  
  Therefore $u\equiv 0.$ This means that $\psi(\Sigma)$ is a surface of revolution around the $x_3-$axis.
 \end{proof}


{\section{Stable CMC hypersurfaces  with free boundary in a slab  in $\r^{n+1}$}
Stability of embedded rotationally invariant CMC hypersurfaces connecting 2 parallel hyperplanes in $\r^{n+1}$ and orthogonal 
to them was studied by Athanassenas \cite{athanassenas}  and Vogel \cite{vogel} for $n=2$, and by Pedrosa and Ritor{\'e} for any $n\geq 2$ \cite{pedrosa-ritore}. It turns out that for $2\leq n\leq 7$ only circular cylinders can be stable. However for $n\geq 9$ there are 
certain unduloids which are stable \cite{pedrosa-ritore}. 

We study here stability of general immersed capillary hypersurfaces. Under a mild condition on the immersion along the boundary, we show that a stable hypersurface  has to be embedded and rotationally invariant around an axis orthogonal to the hyperplanes bounding the slab. Our proof is inspired by some ideas used in \cite{ros-vergasta}. When $n=2,$ the conclusion is valid without any extra condition.  This follows from the results in \cite{ros}. We believe the assumption on the boundary behaviour of the immersion can also be removed for all $n\geq 3.$ We note also that our proof does not use the fact that the hypersurfaces are contained in the slab. Clearly, without loss of generality, the slab can be assumed to be horizontal. More precisely, our result is as follows:

\begin{theorem}\label{thm:free boundary} Let $\psi:\Sigma\to \r^{n+1},\, n\geq 2,$ be an immersed capillary hypersurface connecting  two horizontal hyperplanes in $\r^{n+1}$ with  contact angle $\theta = {\pi}/{2}.$ Suppose that the restriction of $\psi$ to each component of $\partial\Sigma$ is an embedding. 

 If $\psi$ is stable then $\psi(\Sigma)$ is either a circular vertical cylinder or  a vertical graph which is rotationally invariant  around a vertical axis. 
\end{theorem}

\begin{proof}
Call  $\Pi_1$ and $\Pi_2$ the hyperplanes bounding the slab with $\Pi_1$ below $\Pi_2$.  We denote, as usual, by  $e_1,\ldots,e_{n+1}$  the vectors of the canonical basis of $\r^{n+1}$. 
We consider the function $v=\langle N, e_{n+1}\rangle,$ that is, the $(n+1)$-coordinate function of $N.$   If $v$ is identically zero then $\Sigma$ is a vertical cylinder whose base is an   embedded CMC hypersurface in 
$\Pi_1.$ The base is, by Alexandrov's theroem, a round sphere in $\Pi_1$. Consequently, $\Sigma$ is a circular cylinder.

 Assume $v$ is not identically zero. We will show it has a sign in the interior of $\Sigma.$
    Suppose by contradiction  that $v$ changes sign and  consider the functions  $ v_+=\text{max} \{v,0\} $ and
 $v_-= \text{min} \{v,0\}$ which lie in the Sobolev space $H^1(\Sigma).$   We compute, using the fact that $v=0$ on $\partial\Sigma$
 \begin{align*}
 \mathcal I(v_+,v_+)&= \int_{\Sigma} \{\langle \nabla v_+,\nabla v_+\rangle -|\sigma|^2 (v_+)^2\} \,d\Sigma\\
                   &= \int_{\Sigma} \{\langle \nabla v,\nabla v_+\rangle -|\sigma|^2 v v_+\} \,d\Sigma\\
                   &= -\int_{\Sigma} \{(\Delta v +|\sigma|^2 v) v_+\} \,d\Sigma\\
                   &=0
  \end{align*}
 and similarily  $\mathcal I(v_-,v_-)=0$. 
 As we supposed that $v$ changes sign, there exists $a\in \r$ such that $\int_{\Sigma} (v_+ + a v_-)\, d\Sigma =0.$ So we can use  $\widetilde v:= v_+ + a v_-$ as a test function in the second variation. We have
  \begin{equation*}
 \mathcal I(\widetilde v, \widetilde v) =\mathcal  I(v_+,v_+)+2a\mathcal I(v_+,v_-) + a^2 \mathcal I(v_-,v_-)=0.
 \end{equation*}
 Since $\Sigma$ is stable, we conclude that $\widetilde v$ is a Jacobi function and so satisfies $\frac{\partial{\widetilde v}}
 {\partial \nu} = 0$ on $\partial\Sigma.$ Note that $\psi(\Sigma)$ extends analytically by reflection through the hyperplanes $\Pi_1$ and $\Pi_2.$ Since we also have $\widetilde v= 0$ on $\partial\Sigma$, by the uniqueness part in the Cauchy-Kowalevski theorem the function $\widetilde v$ vanishes in a neighborhood of $\partial\Sigma,$ i.e $v$ vanishes in a neighborhood of $\partial\Sigma.$ This means $\psi(\Sigma)$ is a cylinder in a neighborhood of  $\psi(\partial\Sigma)$
 and, by analyticity of CMC hypersurfaces, $\psi(\Sigma)$ is a vertical cylinder and so $v$ is identically zero, a contradiction. 
 
  Therefore the function $v$ does not change sign in $\Sigma$. We will assume $v\geq 0,$ the case $v\leq 0$ being similar. The function $v$ satisfies
 \begin{align*}
v&\geq 0,\\
\Delta v&=-|\sigma|^2 v \leq 0,\\
 v&=0\quad {\text{on}}\quad\partial \Sigma.
\end{align*}
 As we are assuming $v$ is not identically zero, by the maximum principle for superharmonic functions we know that $v>0$ on the interior of $\Sigma.$ So the interior of $\psi(\Sigma)$ is a local vertical graph. We will show $\psi( \Sigma) $ is globally a vertical graph by analyzing its behavior near its boundary. Let $\Gamma_1,\ldots, \Gamma_k$ denote the connected components of $\partial\Sigma.$ By hypothesis, $\psi$ restricted to $\Gamma_i$ is an embedding and so 
 $\psi(\Gamma_i)$ separates  the hyperplane containing it, among $\Pi_1$ and $\Pi_2,$  into two connected components, for each $i=1,\ldots,k.$   
 
Denote by $P: \r^{n+1}\to \Pi_1$ the orthogonal projection and set  $F=P\circ \psi.$  Fix $i=1,\ldots,k$ and consider a point 
 $p\in \Gamma_i$, and a curve
  $\gamma:(-\epsilon,0]\to \Sigma$ parametrized by arc-length  with $\gamma(0)=p$ and $\gamma^\prime(0)=\nu(p)$. It is easy to check that
 \begin{align*}
 \frac{d}{dt}\langle F(\gamma(t))-F(p),N(p)\rangle|_0&=0\\
 \frac{d^2}{dt^2}\langle F(\gamma(t)-F(p),N(p)\rangle|_0&=\langle\frac{D}{dt}\psi(\gamma^\prime)|_0,N(p)\rangle\\
 &=\sigma(\nu,\nu)
  \end{align*}
  Note that 
  $$\frac{\partial v}{\partial \nu}=-\sigma(\nu,\nu)\langle\nu,e_{n+1}\rangle=
  \begin{cases}+\sigma(\nu,\nu)\quad {\text{if}}\quad \psi(\Gamma_i)\subset\Pi_1\\
-\sigma(\nu,\nu)\quad {\text{if}} \quad \psi(\Gamma_i)\subset\Pi_2
\end{cases}
$$
By the strong maximum principle, $\frac{\partial v}{\partial \nu}<0$ on $\partial\Sigma.$
So for $t$ small, $F(\gamma(t))$ lies in the component of $\Pi_1\setminus F(\Gamma_i)$ which has $N(p)$ 
as outwards (resp. inwards) pointing normal at $F(p)$ if $\psi(\Gamma_i)\subset\Pi_1$ (resp. if $\psi(\Gamma_i)\subset\Pi_2$). It follows that there is a thin strip in the interior of $\Sigma$ around $\Gamma_i$ which projects 
on this component. We call $D_i$ the component of  $\Pi_1\setminus F(\Gamma_i)$ which does not intersect this projection.
We define $\widetilde \Sigma$ to be the union of $\Sigma$ with the disjoint union of all the domains $D_i$, and 
$\widetilde F: \widetilde \Sigma \to \Pi_1$ by
\begin{equation*}
\widetilde F =\begin{cases} F\qquad &{\text{on}}\quad \Sigma\\
{\text{projection on}}\, \Pi_1\quad &{\text{on}}\quad D_i, i=1,\ldots,k.
\end{cases}
\end{equation*}

It is clear that $\widetilde F$ is a local homeomorphism and a proper map, thus it is a covering map. Therefore
$\widetilde F$ is a global homeomorphism onto $\Pi_1.$ So $\psi(\Sigma)$ is a graph over a domain in 
$\Pi_1$ and it is, in particular, embedded.  Alexandrov's reflection technique shows that $\psi(\Sigma)$ is a
hypersurface of revolution around a vertical axis, see \cite{wente}.
\end{proof}

{\section{Stable capillary hypersurfaces  in a half-space   in $\r^{n+1}$}

We focus in this section on  capillary immersions into a half-space in $\r^{n+1}$. Spherical caps are examples of such immersions and are actually the only embedded ones \cite{wente}. It is known they are stable and even minimize the energy functional, see \cite{gonzalez et al}.   Marinov  \cite{marinov} characterized the spherical caps as the only stable immersed capillary surfaces in a half-space in $\r^3$ with embedded boundary.  He utilized an infinitesimal admissible variation which is the version in the capillarity setting of the one  used, in the closed case, by Barbosa and do Carmo \cite{barbosa-do carmo}. Recently, Choe and Koiso \cite{choe-koiso} proved the same result in $\r^{n+1}$, for any $n\geq 2$, assuming the  contact  angle  is $\geq \pi/2$ and  the boundary of the hypersurface  is convex.  To achieve this, they computed the second variation of an admissible volume-preserving variation which is the integrated variation of the 
infinitesimal one used by Marinov.  

We  deal here the case where the angle of contact  is $\leq\pi/2.$ We recall that we orient our immersions by their mean curvature vector $\vec H$ and the angle of contact is the one between  $\vec H$  and the exterior unit normal to the boundary of the half-space (note that  by the maximum principle $H\neq 0$).  We will use two infinitesimal variations, the first one being the one used by Marinov,
Choe and Koiso. However, contrarily to the previous authors, thanks to  Proposition \ref{prop:normalintegral}}, to establish that this variation is volume-preserving,  we do not need to assume embeddedness of the boundary. The second infinitesimal variation we use is a suitable combination of the negative and positive parts of the last coordinate of the Gauss map. 
 Our result is the following:

\begin{theorem}\label{thm:halfspace} Let $\psi:\Sigma\to \r^{n+1},\, n\geq 2,$ be a stable  immersed capillary hypersurface in a half-space in $\r^{n+1}$ with  contact angle $0<\theta \leq {\pi}/{2}.$ 

\begin{enumerate}[\sf (i)]
\item If $\theta=\pi/2,$ then $\psi(\Sigma)$ is a hemisphere.
\item  If $\theta<\pi/2$ and the restriction of $\psi$ to each component of $\partial\Sigma$ is an embedding, then 
$\psi(\Sigma)$ 
is a spherical cap.
\end{enumerate}
\end{theorem}

\begin{proof}

Without loss of generality, we may suppose the half-space is the upper half-space  $ x_{n+1}\geq 0.$

 Integrating the equation in (ii) of Proposition \ref{prop:equations},  we get
\begin{equation}\label{eq1}
\int_{\partial \Sigma} \langle \psi,\nu\rangle {d}( \partial\Sigma) = n\int_{\Sigma} \{1+H\langle \psi,N\rangle\}d\Sigma.
\end{equation} 
  
On $\partial\Sigma$, we have:  $\, \cos\theta\, N+\sin\theta\, \nu=-e_{n+1},$ where $e_{n+1}$ is the $(n+1)$-th vector of the canonical basis of $\r^{n+1}$. Therefore, 
Proposition \ref {prop:normalintegral} gives

  \begin{equation}\label{eq:normalintegral}
 n\int_{\Sigma} N\, d\Sigma = -\frac{1}{\cos\theta}\left( \int_{\partial\Sigma} \langle \psi,\nu\rangle\, \,d(\partial\Sigma)\right)\, e_{n+1}.
  \end{equation}

  From (\ref{eq1}) and (\ref{eq:normalintegral}), we conclude that:
\begin{equation*}
\int_{\Sigma} \{1+H\langle \psi, N\rangle +\cos\theta \langle N,e_{n+1}\rangle\} \,d\Sigma =0.
\end{equation*}
So we may use  $\phi:=1+H\langle \psi, N\rangle +\cos\theta \langle N,e_{n+1}\rangle$ as a test  function in the stability inequality. Set $u=\langle \psi,N\rangle$ and $v=\langle N,e_{n+1}\rangle.$ From Proposition \ref{prop:equations} we know that:
\begin{equation}\label{eq3}
\Delta u +|\sigma|^2 u =-nH,
\end{equation}
and 
\begin{equation}\label{eq4}
\Delta v+ |\sigma|^2 v =0.
\end{equation}
Using these equations, we  compute:
\begin{align*}
\Delta \phi&= H(-nH-|\sigma|^2 u)+\cos\theta \, |\sigma|^2 v\\
&=-nH^2-|\sigma|^2(H u-\cos\theta  v)\\
&=-nH^2-|\sigma|^2(\phi-1).
\end{align*}
Therefore
\begin{equation*}
\phi\Delta\phi+ |\sigma|^2\phi^2 =(|\sigma|^2-nH^2)\phi.
\end{equation*}
On the other hand:
\begin{equation*}
\frac{\partial u}{\partial \nu}= \langle \nu, N\rangle + \langle \psi, D_{\nu} N\rangle=-\sigma(\nu,\nu)\langle \psi,\nu\rangle\\
\end{equation*}
and 
\begin{equation*} 
\frac{\partial v}{\partial \nu}= \langle D_{\nu}N, e_{n+1}\rangle=-\sigma(\nu,\nu)\langle\nu,e_{n+1}\rangle=\sigma(\nu,\nu)\sin\theta.
\end{equation*}
Using the relation: $-e_{n+1}= \cos\theta \,N+\sin\theta\,\nu,$  one can check  after direct computations that:
\begin{equation}\label{eq:normalderivative}
\frac{\partial \phi}{\partial \nu}= \cot\theta\, \sigma(\nu,\nu) \,\phi.
\end{equation}
It follows that
\begin{align}\label{eq7}
\mathcal I(\phi,\phi)=&-\int_{\Sigma} (|\sigma|^2-nH^2)\phi\, d\Sigma.
\end{align}
Integrating (\ref{eq3}) we have
\begin{equation*}
\int_{\partial\Sigma} \frac{\partial u}{\partial \nu}  \,d(\partial\Sigma) + \int_{\Sigma}|\sigma|^2 u \,d\Sigma = -nH\int_{\Sigma} 1 \,d\Sigma.
\end{equation*}

Using this, we can write
\begin{align*} \int_{\Sigma} (|\sigma|^2-nH^2)\phi\, d\Sigma =& \int_{\Sigma} (|\sigma|^2-nH^2)\,d\Sigma
+ H \int_{\Sigma} |\sigma|^2 u \,d\Sigma -nH^2\int_{\Sigma} H u \,d\Sigma\\
&+\cos\theta \int_{\Sigma} (|\sigma|^2-nH^2)v \,d\Sigma\\
=&\int_{\Sigma} (|\sigma|^2-nH^2)\,d\Sigma -nH^2\int_{\Sigma} \left(1+H u\right) \,d\Sigma -H\int_{\Sigma}
\frac{\partial u }{\partial \nu} \ \,d\Sigma\\
& +\cos\theta \int_{\Sigma} (|\sigma|^2-nH^2) v \,d\Sigma\\
=&\int_{\Sigma} (|\sigma|^2-nH^2)\,d\Sigma +nH^2\cos\theta\int_{\Sigma} v\,d\Sigma - H\int_{\partial\Sigma}\frac{\partial u}{\partial \nu} \,d(\partial \Sigma)\\
& +\cos\theta \int_{\Sigma} (|\sigma|^2-nH^2) v \,d\Sigma\\
=&\int_{\Sigma} (|\sigma|^2-nH^2)\,d\Sigma +\cos\theta \int_{\Sigma} |\sigma|^2 v \,d\Sigma
- H\int_{\partial\Sigma}\frac{\partial u}{\partial \nu} \,d(\partial \Sigma).
\end{align*}
Integrating  (\ref{eq4}) we obtain 
\begin{equation*}
\int_{\Sigma} |\sigma|^2 v\,d\Sigma=-\int_{\partial\Sigma} \frac{\partial v}{\partial \nu} \,d({\partial\Sigma}).
\end{equation*}
So,
\begin{equation*}
\int_{\Sigma} (|\sigma|^2-nH^2)\phi\, d\Sigma =\int_{\Sigma} (|\sigma|^2-nH^2)\,d\Sigma -\int_{\partial\Sigma}
\frac{\partial}{\partial \nu}\left(Hu+\cos\theta\, v\right)\,d\Sigma,
\end{equation*}
that is,
\begin{equation}\label{eq8}
 \int_{\Sigma} (|\sigma|^2-nH^2)\phi\, d\Sigma=\int_{\Sigma} (|\sigma|^2-nH^2)\,d\Sigma -\int_{\partial\Sigma}
\frac{\partial\phi}{\partial \nu}\,d(\partial\Sigma).
\end{equation}

Let $i\in\{1,\ldots ,k\}$ and $\{v_1,\ldots,v_{n-1}\}$ be a local orthonormal frame on  $\partial\Sigma.$ Then:
\begin{equation*} 
\sigma(\nu,\nu)=nH-\sum_{j=1}^{n-1} \sigma(v_j,v_j).
\end{equation*}
Now, considering the unit normal $\bar\nu$ in $\r^n\times\{0\}$ to  $\partial\Sigma$ along $\psi$, as chosen in Sec.\ \ref{sub:variational}, we have $ N=-\sin\theta\,\bar\nu - \cos\theta \, e_{n+1}.$ We can thus write 
\begin{equation*}\sigma(v_j,v_j)=-\langle \nabla_{v_j} N, v_j\rangle=\sin\theta\, \langle \nabla_{v_j} \bar\nu, v_j\rangle.
\end{equation*}
Therefore, if we denote by $H_{\partial\Sigma}$ the mean curvature of $\partial\Sigma$ in $\r^n\times\{0\}$ computed with respect to the unit normal 
$\bar\nu,$ the following  relation holds on $\partial\Sigma$
\begin{equation}\label{eq:sigma}
\sigma(\nu,\nu)=nH+(n-1)\sin\theta\, H_{\partial\Sigma}.
\end{equation}
Also, taking into account that $\nu=\cos\theta\,\bar\nu-\sin\theta\,e_{n+1}$, one has on $\partial\Sigma$
\begin{equation*}
\phi=1-\sin\theta\, H \langle\psi,\bar\nu\rangle -\cos^2\theta=\sin^2\theta - \sin\theta\, H\langle\psi,\bar\nu\rangle.
\end{equation*}
Replacing in (\ref{eq:normalderivative}), we obtain:
\begin{equation*}
\frac{\partial \phi}{\partial \nu}=\cos\theta\{ nH\sin\theta -n H^2\langle \psi,\bar\nu\rangle +(n-1)\sin^2\theta H_{\partial\Sigma}\, -(n-1)\sin\theta\,H_{\partial\Sigma}\ H\,\langle \psi,\bar\nu\rangle\}
\end{equation*}
and so, using   (ii) of Proposition \ref{prop:equations} applied to the immersion $\psi_{|\partial\Sigma}:\partial\Sigma\to\r^n\times\{0\}$, we get
\begin{multline}\label{eq:integralderiv}
 \int_{\partial\Sigma} \frac{\partial \phi}{\partial \nu}\,d(\partial\Sigma) =nH\cos\theta\left[\sin\theta\, \text{vol}_{n-1}(\partial\Sigma) -H\int_{\partial\Sigma} \langle \psi, \bar\nu\rangle\,d(\partial\Sigma)\right]\\
 +(n-1)\cos\theta\sin\theta\left[ H \, \text{vol}_{n-1}(\partial\Sigma)
  +\sin\theta \int_{\partial\Sigma}H_{\partial\Sigma}\,d(\partial\Sigma)\right].
\end{multline}
Integrating equation (i)  of Proposition \ref{prop:equations}, we get:
\begin{equation*}
\int_{\partial\Sigma} \nu\, d(\partial\Sigma)= nH\int_{\Sigma} N\,d\Sigma.
\end{equation*}
Therefore
\begin{equation*}
\int_{\partial\Sigma} \langle \nu,e_{n+1}\rangle \, d(\partial\Sigma)= nH\int_{\Sigma}\langle  N,e_{n+1}\rangle\,d\Sigma,
\end{equation*}
that is,
\begin{equation}\label{eq5}
-\sin\theta\,\text{vol}_{n-1}(\partial\Sigma)= nH\int_{\Sigma} \langle  N,e_{n+1}\rangle\,d\Sigma.
\end{equation}
Combining (\ref{eq5}), (\ref{eq:normalintegral}) and taking into account the relation $\langle \psi,\nu\rangle=\cos\theta\langle \psi, \bar\nu\rangle$, we conclude that
\begin{equation*}
\sin\theta\, \text{vol}_{n-1}(\partial\Sigma) =H\int_{\partial\Sigma} \langle \psi, \bar\nu\rangle\,d(\partial\Sigma).
\end{equation*}
Therefore (\ref{eq:integralderiv}) becomes
\begin{equation*}
 \int_{\partial\Sigma} \frac{\partial \phi}{\partial \nu}\,d(\partial\Sigma) =(n-1)\sin\theta\cos\theta\left[H\,{\text{vol}_{n-1}} (\partial\Sigma)+
  \sin\theta  \int_{\partial\Sigma}H_{\partial\Sigma} \, d(\partial\Sigma) \right].
\end{equation*}
Combining this with  (\ref{eq7}) and (\ref{eq8}), we finally obtain 
\begin{multline}\label{eq:I of phi}
\mathcal I(\phi,\phi)=-\int_{\Sigma}\{ |\sigma|^2-nH^2\}d\Sigma\\ +
(n-1)\sin\theta\cos\theta\left[H\,{\text{vol}_{n-1}} (\partial\Sigma)+
  \sin\theta \int_{\partial\Sigma}H_{\partial\Sigma} \, d(\partial\Sigma) \right].
\end{multline}
By stability $\mathcal I(\phi,\phi)\geq 0$. In particular, if $\theta =\pi/2$, since $|\sigma|^2\geq nH^2$, we conclude that necessarily $|\sigma|^2=nH^2$ everywhere on $\Sigma$. So the immersion is totally umbilical, that is, $\psi(\Sigma)$ is a hemisphere. This proves (i).

To prove (ii), we will first show that the function   $v=\langle N, e_{n+1}\rangle$, does not change sign on $\Sigma$. Suppose, by contradiction, this is not the case. Then, we can find a  number $\alpha\in \r$ so that the function $w:=v_{-}+\alpha v_{+}$ satisfies
$\int_{\Sigma} w\, d\Sigma = 0.$
Under the hypothesis $\theta <\frac{\pi}{2},$ we have $w=v_{-}=-\cos\theta$ on $\partial\Sigma.$ 
Proceeding as in the proof of Theorem \ref{thm:free boundary} and using that $\frac{\partial v}{\partial\nu}=\sin\theta\,\sigma(\nu,\nu)$, we find that
\begin{align*}
\mathcal I(v_+,v_+)&=\mathcal I(v_-,v_+)=0\\
\mathcal I(v_-,v_-)&= \int_{\partial\Sigma} \left(\frac{\partial v}{\partial\nu}-\cot\theta\,\sigma(\nu,\nu)  v\right)v\,d(\partial\Sigma)\\
                &=-\cot\theta\int_{\partial\Sigma}\sigma(\nu,\nu)\,d(\partial\Sigma).\\
\end{align*}
 Taking into account (\ref{eq:sigma}), we get
 \begin{align*}\label{I of w}
\mathcal I(w,w)&=\mathcal I(v_-,v_-)+2\alpha \mathcal I(v_-,v_+)+\alpha^2 \mathcal I(v_+,v_+)\\
&= -nH\, \cot\theta\, {\text{vol}_{n-1}}(\partial\Sigma) -(n-1) \cos\theta\int_{\partial\Sigma} H_{\partial\Sigma} \, d(\partial\Sigma).
\end{align*}
By stability 
\begin{equation*}
\mathcal I(\phi,\phi)+\sin^2\theta\, \mathcal I(w,w)\geq 0.
\end{equation*}
Taking into account (\ref{eq:I of phi}), we conclude that 
\begin{equation*}
-\int_{\Sigma}\left[ |\sigma|^2-nH^2\right]\,d\Sigma -\sin\theta\cos\theta\, H\,{\text{vol}_{n-1}}(\partial\Sigma)\geq 0.
\end{equation*}
As $|\sigma|^2\geq nH^2,$ $\theta <\pi/2$ and $H> 0$ by our choice of orientation, this is a contradiction (recall that  $H\neq 0$  because of the maximum principle ). We conclude that  the function $v$ does not change  sign on $\Sigma.$ Since $ v=-\cos\theta<0$ on $\partial\Sigma,$ we must have $v\leq0$ everywhere on  $\Sigma.$  
As $\Delta v=-|\sigma|^2 v\geq 0,$  by the Hopf maximum principle we have $v<0$ everywhere.  So $\psi(\Sigma)$ is a local vertical graph around each of its  points, including the boundary ones.  We can now conclude, as in the proof of Theorem \ref{thm:free boundary}, that $\psi(\Sigma)$ is globally a graph over a domain in 
$\r^n\times\{0\}$, hence it is embedded and a spherical cap by Alexandrov's reflection argument \cite{wente}.
\end{proof}

\begin{remark} 

The case $\theta =\pi/2$ can be treated more directly. Indeed, by  reflection  through the boundary hyperplane, the hypersurface gives rise to a closed CMC hypersurface which can be shown to be stable for volume-preserving variations and which is therefore a round sphere by the result of Barbosa and do Carmo \cite{barbosa-do carmo}. 
\end{remark}

{\bf Acknowledgments.} {\small 
This work was carried out during the first  author's visit at the G{\'e}om{\'e}trie et Dynamique group of IMJ-PRG-Universit{\'e} Paris Diderot. He is grateful  to  this group for its hospitality.}



\begin{thebibliography}{12}


\bibitem{alarcon-souam}  A. Alarc\'{o}n and R. Souam.:  {\em Capillary surfaces inside polyhedral regions.}   arXiv:1401.6935. 

\bibitem{athanassenas} M. Athanassenas.: {\em 
A variational problem for constant mean curvature surfaces with free boundary. }
J. Reine Angew. Math. {\bf 377} (1987), 97--107. 

\bibitem{barbosa-do carmo} J-L. Barbosa and M. do Carmo.: {\em  Stability of hypersurfaces with constant mean curvature.}  Math. Z. {\bf 185} (1984), no. 3, 339--353.

\bibitem{barbosa et al} J-L. Barbosa; M.  do Carmo and J. Eschenburg.:  {\em 
 Stability of hypersurfaces of constant mean
curvature in Riemannian manifolds. } Math. Z.  {\bf 197} 
(1988),  no. 1, 123--138.


\bibitem{cheng} S. Y. Cheng.:
 {\em  Eigenfunctions and nodal sets.}
Comment. Math. Helv.  {\bf 51}  (1976), no. 1, 43--55.


\bibitem{choe-koiso} J. Choe and  M. Koiso.:  {\em Stable capillary hypersurfaces in a wedge.}  arXiv:1405.5407. 



\bibitem{Fi} R. Finn.:  {\em Equilibrium capillary surfaces. }  Grundlehren der Mathematischen Wissenschaften, 284. Springer--Verlag, New York, 1986. 

\bibitem{fraser-schoen} A. Fraser and R. Schoen.: {\em Sharp eigenvalue bounds and minimal surfaces in the ball. } arXiv:1209.3789. 


\bibitem{gilbarg-trudinger} D. Gilbarg and N. Trudinger.:  {\em Elliptic partial differential equations of second order. }Reprint of the 1998 edition. Classics in Mathematics. Springer-Verlag, Berlin, 2001.


\bibitem{gonzalez et al} E. Gonzalez,  U.  Massari, U and I. Tamanini.: {\em 
Existence and regularity for the problem of a pendent liquid drop. } 
Pacific J. Math. {\bf 88} (1980), no. 2, 399Ð420. 

\bibitem{langbein} D. Langbein.:  {\em Capillary Surfaces. Shape--stability--dynamics, in particular under weightlessness. }Springer--Verlag, Berlin, 2002.

\bibitem{li-xiong} H. Li and C. Xiong.: {\em Stability of capillary hypersurfaces in a Euclidean ball.} arXiv:1408.2086.

 
 
\bibitem{marinov} P. Marinov.:  {\em Stability of capillary surfaces with planar boundary in the absence of gravity.}  Pacific J. Math. {\bf 255} (2012), no. 1, 177--190.


\bibitem{Mc} J. McCuan.:  {\em  Symmetry via spherical reflection and spanning drops in a wedge.} Pacific J. Math. {\bf 180} (1997), no. 2, 291--323.

\bibitem{nitsche}  J.C.C. Nitsche.:  {\em  Stationary partitioning of convex bodies. } Arch. Rational Mech. Anal. {\bf 89} (1985), no. 1, 1Ð19. 


\bibitem{osserman} R. Osserman.:  {\em Circumscribed circles. } Amer. Math. Monthly {\bf 98}  (1991), no. 5, 419--422. 

\bibitem{Pa} S. -h. Park.:  {\em Every ring type spanner in a wedge is spherical. } Math. Ann. {\bf 332} (2005), no. 3, 475--482.

\bibitem{pedrosa-ritore} R. Pedrosa and M. Ritor{\'e}.:  {\em Isoperimetric domains in the Riemannian product of a circle with a simply connected space form and applications to free boundary problems. }
Indiana Univ. Math. J. {\bf 48 }(1999), no. 4, 1357--1394.

\bibitem{ros1} A. Ros.: {\em Stability of minimal and constant mean curvature surfaces with free boundary.} Mat. Contemp. {\bf 35} (2008), 221--240. 

\bibitem{ros} A. Ros.: {\em Stable periodic constant mean curvature surfaces and mesoscopic phase separation.} Interfaces Free Bound. {\bf 9} (2007), no. 3, 355--365. 

\bibitem{ros-vergasta} A. Ros and E. Vergasta.: {\em Stability for hypersurfaces of constant mean curvature with free boundary.} Geom. Dedicata {\bf 56}  (1995), no. 1, 19--33.

\bibitem{ros-souam} A. Ros and R. Souam.:  {\em On stability of capillary surfaces in a ball.}  
Pacific J. Math. {\bf 178} (1997), no. 2, 345--361.


\bibitem{souam} R. Souam.: {\em Schiffer's problem and an isoperimetric inequality for the first buckling eigenvalue of domains on $\s^2$.}  Ann. Global Anal. Geom. {\bf 27} (2005), no. 4, 341--354.


\bibitem{vogel} T. I.  Vogel.:  {\em Stability of a liquid drop trapped between two parallel planes.}  SIAM J. Appl. Math. {\bf 47} (1987), no. 3, 516--525.

\bibitem{vogel2} T. I. Vogel.: {\em Stability of a liquid drop trapped between two parallel planes. II. General contact angles.} SIAM J. Appl. Math. {\bf 49}  (1989), no. 4, 1009--1028.

\bibitem{wente} H. C. Wente.:  {\em The symmetry of sessile and pendent drops.}   Pacific J. Math. {\bf 88} (1980), no. 2, 387--397.

\bibitem{We} H. C. Wente.:  {\em Tubular capillary surfaces in a convex body.}  Advances in geometric analysis and continuum mechanics (Stanford, CA, 1993), 288--298, Int. Press, Cambridge, MA, 1995.

\bibitem{zhou} L. Zhou.:  {\em On stability of a catenoidal liquid bridge. }
Pacific J. Math. {\bf 178} (1997), no. 1, 185--198.

\end{thebibliography}
\end{document}